\documentclass[11pt,reqno]{amsart}
\usepackage{amscd,amssymb,amsmath,amsthm}
\usepackage[arrow,matrix]{xy}
\usepackage{graphicx}
\usepackage{cite}
\usepackage{color}
\newtheorem{thm}[subsection]{Theorem}

\newtheorem{pro}[subsection]{Proposition}
\newtheorem{cor}[subsection]{Corollary}

\newtheorem{rk}[subsection]{Remark}
\newtheorem{defn}[subsection]{Definition}

\numberwithin{equation}{section} 

\newcommand{\bea}{\begin{eqnarray}}
\newcommand{\eea}{\end{eqnarray}}

\newcommand{\R}{\mathbb{R}}

\def \> {\Rightarrow}
\def \0 {\emptyset}

\begin{document}
\title[$(G,\mu)-$Quadratic stochastic operators]{$(G,\mu)-$Quadratic stochastic operators}

\author{J. Blath, U.U. Jamilov, M. Scheutzow}

\address{M.\ Scheutzow and J.\ Blath\\ Institut f\"ur Mathematik, MA 7-5, Fakult\"at II,
        Technische Universit\"at Berlin, Stra\ss e des 17.~Juni 136, 10623 Berlin, FRG;}
\email {ms@mail.math.tu-berlin.de\   \ blath@mail.math.tu-berlin.de}
  \address{U.\ U.\ Jamilov\\ Institute of mathematics at the National University of Uzbekistan,
29, Do'rmon Yo'li str., 100125, Tashkent, Uzbekistan.}
\email {jamilovu@yandex.ru}

\date{\today}

  \maketitle

\begin{abstract}
     We consider a new subclass of quadratic stochastic (evolutionary) operators on the simplex indexed by a finite Abelian  group $G$ 
with heredity law $\mu$. With the help of the notion of $s(\mu)$-invariant subgroups, where $s(\mu)$ denotes the support of $\mu$ in $G$, 
we prove that almost all (w.r.t.\ Lebesgue measure) trajectories of such operators converge to a unique fixed point which is the center 
of the simplex. We also identify and describe the periodic trajectories of the operator and  give conditions for regularity and
periodicity.
\end{abstract}
\vskip 0.5 truecm
\maketitle


\vskip 0.5 truecm


\section{Introduction}

The notion of quadratic stochastic operators (QSOs) was introduced by S.N.~Bernstein in \cite{Br}, and since then the theory of quadratic 
stochastic operators has been developed for more than 85 years (see e.g.\ \cite{RN1} --\cite{GZ}, \cite{HS},\cite{K1},\cite{K2},\cite{RJ1},\cite{RJ2} for some classic as well as recent results).  While QSOs were originally introduced as ``evolutionary operators'' describing the dynamics of gene frequencies for given laws of heredity in mathematical population genetics (see \cite{Lyu} for a comprehensive account), they are also interesting from a purely mathematical point of view. Their full classification remains a challenging open problem. 


A quadratic stochastic {\color{black}(evolutionary)} operator {\color{black}arises in population genetics} as follows. Consider a {(large)} population { with $m \in \mathbb N$ different genetic types. Let $[m]:=\{1, 2, \dots, m\}$ and} $x^0 = (x^0_1,...,x^0_m)$ be the {relative frequencies of the genetic types within the whole population in the present generation, which is a probability distribution and hence an element of the simplex indexed by $[m]$ which we denote by $S^{m-1}$. To determine  the (expected) gene frequencies in the next generation, let} $p_{ij,k}$ {be} the probability
that {two} individuals {of type $i$ resp.\ $j$} interbreed to
produce {an offspring with genetic type} $k$. Then, the probability distribution {$x' = (x'_1,...,x'_m) \in S^{m-1}$ describing the (expected) gene frequencies} in the
{next} generation {is given by} 
\begin{equation}
\label{ng}
x'_k =\sum\limits_{i,j=1}^m p_{ij,k}x^0_ix^0_j, \  \  k = 1,...,m.
\end{equation}

The association $x^0 {\mapsto} x'$ defines a map
{$V: S^{m-1} \to S^{m-1}$} called  evolution{ary} operator. The population evolves by
starting from an arbitrary {frequency distribution} $x^0$, then passing to the state
$x' = V (x^0)$ (the next  {``}generation"), then to the state
$x'' = V (V (x^0))$, and so on.
Thus {the evolution of gene frequencies} of the population {can be considered as a} dynamical system
$$
x^0, \  \ x' = V (x^0), \  \ x'' = V^2(x^0), \  \ x'''= V^3(x^0), \  \ . . .
$$
Note that {$V$} (defined by (\ref{ng}){)} is a non-linear (quadratic) operator, and i{ts dimension increases with $m$}. Higher dimensional dynamical systems are important but there are
relatively few dynamical phenomena that are currently understood
(\!\!\cite{D}, \cite{E}, \cite{R}).\\

{One of the main objects of study for QSOs is the asymptotic behavior of their
trajectories depending on the initial value}. 
This { has been determined so far only for certain special subclasses of QSOs.}
{ Indeed, a natural choice for the $p_{ij, k}$, also called ``coefficients of heredity'', with an obvious biological interpretation, is given by} 
\begin{equation}
\label{koefvolt}
p_{ij,k}=0, \ \ \mbox {if} \ \ k\notin \{i,j\}, \ \ i,j,k=1,...,m.
\end{equation}
{ In this case, we speak of a {\em Volterra QSO}, and the corresponding asymptotic behaviour of their trajectories has been analysed in \cite{RN1}, \cite{RN2} and \cite{RNEs}}
using the theory of Lyapunov functions
and tournaments. In \cite{MAT}, {infinite dimensional Volterra operators and their dynamics have been studied.}

{However, in the non-Volterra case (i.e., where condition (\ref{koefvolt}) is violated), many questions remain open and there seems to be no general theory 
available.} 
See \cite{GMR} for a recent review of QSOs.\\ 

{
In the present article, we investigate a certain class of non-Volterra QSOs which exhibit an additional group structure in the definition of the $p_{ij,k}$ that allows us to obtain rather complete asymptotic results. More precisely, instead of considering the simplex $S^{m-1}$ over $[m]$, we regard a finite Abelian group $G=(G, +)$, say of order $m$, and the corresponding simplex $S^G$ indexed by $G$ (which can also be regarded as the space of all probability measures on $G$). Then, for any measure $\mu \in S^G$ we define
the ``coefficients of heredity'' by}
\begin{equation*}
p_{ij,k}{:}=\mu_{k-i-j},\quad  \forall i,j,k \in G.
\end{equation*}
{If $s(\mu)$ denotes the support of $\mu$ in $G$, we introduce the notion of an ``$s(\mu)-$invariant subgroup'', with the help of which} we prove that the trajectory of such operator{s} always converges either to {a} periodic trajectory or to {a fixed point. In particular, they are all ergodic}. We also show that the speed of convergence {to the limit resp.\ to the periodic orbit is rather} fast ({in fact, double-exponential).
Finally, we give criteria for regularity and periodicity.}
\\

{Note that i}n \cite{GWZ}{, the authors also} consider {a} class of quadratic stochastic operators corresponding to {a finite} Abelian group, {however with a different choice of the $p_{ij,k}$.} {We will discuss their model and result below.}\\

The paper is organized as follows. The next chapter provides some preliminaries and {previously known} results from the theory
of QSO{s}. In {C}hapter 3 we introduce {our} new class {of} nonlinear operators and {state and prove our results.}

\section{Preliminaries {and known results}}
A quadratic stochastic operator (QSO) 
  is a mapping {$V$} of the simplex
\begin{equation}\label{simp}
S^{m-1}={\Big\{}x=(x_1,...,x_m)\in {\mathbb{R}^m}: x_i\geq 0, \, \sum^m_{i=1}x_i=1 {\Big\}}
\end{equation}
into itself, of the form {$V(x)=x' \in S^{m-1}$ where}
\begin{equation}
\label{kso}
x_k'=\sum^m_{i,j=1}p_{ij,k}x_ix_j, \ \ (k=1,...,m),
\end{equation}
{and the} $p_{ij,k}$  {satisfy}
\begin{equation}\label{koefkso}
p_{ij,k}=p_{ji,k}\geq 0, \quad  \ \ \sum^m_{k=1}p_{ij,k}=1, \ \ (i,j,k=1,...,m).
\end{equation}
The trajectory (orbit) $\{x^{(n)}\}$ for an initial {value} $x^{(0)}\in S^{m-1}$  is defined by 
$$
x^{(n+1)}=V(x^{(n)})={V^{n+1}(x^{(0)})}, \quad n=0,1,2,\dots
$$ 

{We now} recall some definitions and results from the theory {of} QSO{s}.  

\begin{defn}
\label{fp}
A  point $x\in S^{m-1}$ is called a fixed point of a QSO $V$ if $V(x)=x.$
\end{defn}

\begin{defn}
\label{reg}
A {QSO} $V$ is called regular if for any initial point $x \in S^{m-1}$  the limit
$$\lim_{n\rightarrow \infty}V{^n(x)}$$
exists.
{We call the  $V$ almost regular, if the above condition holds for almost all (w.r.t.\ Lebesgue-measure) initial points $x$.}
\end{defn}

Note that {our QSOs are continuous operators and that the simplex over a finite set is compact and convex, so that by the Brouwer Fixed-Point Theorem there is always at least one fixed point. Further, by continuity, any limit point of a QSO is also a fixed point.}
Limit behavior of trajectories and fixed points of QSO{s} play {an} important role in many applied problems{, see e.g.} \cite{RN1},\cite{RN2},\cite{J},\cite{K1}, \cite{Lyu}.
The {intuitive meaning} of the regularity of {a} QSO  {in terms of mathematical genetics} is {obvious: In the} long run the distribution of {gene frequencies tends to an equilibrium, no matter what the initial condition was. Further, if a given 
limit point is strictly inside the simplex, then this means that there is long-term coexistence of genetic types under the respective initial distribution.

Of course, it is not necessarily clear  whether a given set of ``coefficients of heredity'' has a direct biological interpretation at all, but this is not the main point of this paper. Instead, as mentioned above, we focus on a certain part of the classification problem for QSOs that is inspired a priori purely by mathematical curiosity.
}

\begin{defn}
\label{def:erg}
$V$ is said to be ergodic if the limit
\begin{equation}
\label{erg}
\lim_{n\rightarrow \infty } \frac{1}{n} \sum_{k=0}^{n-1} V^k(x)
\end{equation}
exists for any $x\in S^{m-1}.$
\end{defn}

Evidently, any regular QSO and -- more generally -- any QSO for which every trajectory converges to a (not necessarily strict) periodic orbit 
is ergodic, but the converse is not necessarily true.\\

On the basis of numerical calculations Ulam conjectured \cite{U} that
any QSO is ergodic.
  In 1977, Zakharevich \cite{Z} proved that this conjecture is false in general.
  Later in \cite{GZ} necessary and sufficient
condition{s for ergodicity of a QSO defined on $S^2$ were established}.

Now we {introduce} some notation and results from  \cite{GWZ}.
{Recall} that the simplex $S^{m-1}$ is the set of all probability measures on ${[m]}=\{1,...,m\}$.
We {now} consider instead of ${[m]}$ { a finite Abelian group $G$} of order $m$ {and the corresponding simplex $S^G$ over $G$}.
Let $U \subset G$ be a subgroup of $G$ and $\{g+U: g\in G\}$ be the cosets of $U$ in $G$.
Suppose $\lambda\in {S^G}$ is a fixed positive measure, that is $\lambda_i=\lambda(i)>0$ for any $i\in G$.
Then we define the coefficients of heredity {as in \cite{GWZ} by}
\begin{equation} 
\label{koefGZ}
p_{ij,k}=\left\{\begin{array}{lll}
\frac{\lambda_k}{\lambda(i+j+U)}, \ \ \mbox {if} \ \ k\in \{i+j+U\};\\[3mm]
0, \ \ \mbox {otherwise}, \ \
i,j,k \in G.
\end{array}\right.
\end{equation}

{Note that if} $U=\{0\}$, { where $0$ is the} identity element of {the} group $G$,  then {the} corresponding QSO has the form
\begin{equation}
\label{ksoGZ}
x'_k=\sum\limits_{i,j\in G:\atop i+j=k} x_ix_j
\end{equation}
{For convenience, we will freely use the obvious analogs of Definitions (\ref{fp}), (\ref{reg}) and (\ref{def:erg}) for QSO on the simplex $S^G$ instead of $S^{m-1}$.}

\begin{thm}{\bf\cite{GWZ}}\label{GZ}
Almost all {(w.r.t\ Lebesgue measure)} orbits of {the} QSO {defined by} (\ref{ksoGZ}) {converge} to the center of the simplex.
\end{thm}

\begin{cor}
The QSO (\ref{ksoGZ}) is almost regular.
\end{cor}

\section{{Asymptotic behaviour of} $(G,\mu)-$quadratic stochastic operator{s}}

Let $(G,+)$ be a finite Abelian group, $|G|=m$ and {$S^G$ be} the set of all probability measures on $G$, where $|.|$ denotes the cardinality of a set.
We denote the identity element of $G$ by 0.
Let $\mu\in {S^G}$ be {a} fixed measure. Then, {we define} coefficients of heredity by
\begin{equation}
\label{koefmu}
p_{ij,k}:=\mu_{k-i-j},\quad  \forall i,j,k \in G. 
\end{equation}
It is easy to check that for arbitrary $i,j,k \in G$ the conditions (\ref{koefkso}) are
satisfied.

\begin{defn}
The QSO satisfying (\ref{kso}), (\ref{koefkso}) and (\ref{koefmu}) is called $(G,\mu)-$quadratic stochastic
operator.
\end{defn}

Any $(G,\mu)-$QSO has the form

\begin{equation}
\label{Gmkso}
V:x'_k=\sum_{i,j \in G} p_{ij,k}x_ix_j=\sum_{i \in G} \sum_{l \in G}\mu_lx_ix_{k-l-i}.
\end{equation}


\begin{rk}
If $\mu_0=1$ then the corresponding ${(}G,\mu{)}-$QSO coincides with {the} QSO {obtained from} (\ref{ksoGZ}).
\end{rk}

\begin{defn}
Let $U$ be a subgroup of $G$ and $A \subseteq G$ a nonempty set. {Then,} $U$ is called {\em $A$-invariant} if $|U+A|=|U|$.
\end{defn}

{Recall that a {\em coset} of a subgroup $U$ in an Abelian group $G$ is of the form $\{g+u, u \in U\}$ for some $g \in G$.} It is easy {to derive the following basic properties of $A$-invariant sets}.
\begin{pro}
Let $U$ be any subgroup of $G$ {and $A \subset G$}. 
\begin{itemize}
\item $U$ is $A$-invariant iff $A$ is contained in a coset of $U$.
\item If $U$ is $A$-invariant and $\emptyset \neq \tilde A \subset A$, then $U$ is  $\tilde A$-invariant.
\item $|A|=1$ implies that $U$ is $A$-invariant.
\item If $|A|>m/2$, then the only $A$-invariant subgroup is $U=G$.
\end{itemize}
\end{pro}


\begin{pro}
\label{veryeasy}
{Let $B$ be a non-empty subset of $G$. Then, $|B+B|=|B|$ iff $B$ is a coset of some subgroup of $G$.}
\end{pro}

\begin{proof}
{If $B$ is coset of $G$, then it can be written as $g+U$ for some subgroup $U$,
and the first part of the result follows from
$$
|B+B|=|g+g+U+U|=|g+g+U|=|B|.
$$
Conversely let $B=\{b_1, b_2, \dots, b_k\}$. Then $B=b_1 + \tilde B$, where
$\tilde B=\{0, b_2-b_1, \dots, b_k-b_1\}$. With this notation, we obtain
$$
B+B=(b_1+b_1)+\tilde B + \tilde B \supseteq (b_1+b_1) + \tilde B
$$
since $0 \in \tilde B$. If $|B+B| = |B|=k$, then equality holds and we have $\tilde B+\tilde B= \tilde B$, so $\tilde B$ is a subgroup and $B$ is a coset of $\tilde B$.} 
\end{proof}

{For $x \in S^G$} denote by $s(x)=\{i \in G: x_i>0\}$ the {\it support} of $x$. {We now investigate whether the cardinality of support $s(x)$ of a state $x$ grows after an application of $V$, depending on the support $s(\mu)$ of $\mu$
(where, by slight abuse of notation, we do not distinguish between the support of a ``point'' $x\in S^G$ and a ``measure'' $\mu \in S^G$).}

\begin{pro}
{ Let $V$ be a $(G, \mu)-$QSO. Then:}
\begin{itemize}
\item[\rm{a)}] For any $x \in {S^G}$, we have $|s(Vx)|\ge |s(x)|$.
\item[\rm{b)}] For $x \in {S^G}$ we have $|s(Vx)|=|s(x)|$ iff $s(x)$ is equal to a coset of an ${s(\mu)}$-invariant subgroup.
\end{itemize}
\end{pro}

\begin{proof}
a) {F}rom (\ref{Gmkso}) we have 
\begin{equation}
\label{itrsupp}
s(Vx)=s(x)+s(x)+{s(\mu)},
\end{equation}
so a) holds.\\
b) First assume that $s(x)=g+U$ and $U$ is an ${s(\mu)}$-invariant subgroup. Then 
\begin{equation}
\label{eq:gUgU}
s(Vx)=g+U+g+U+{s(\mu)}=(g+g)+U+{s(\mu)},
\end{equation} 
so {that}
$$|s(Vx)|=|(g+g)+U+{s(\mu)}|=|U|=|s(x)|.$$
Conversely, assume that $|s(Vx)|=|s(x)|$, i.e.~$|s(x)+s(x)+{s(\mu)}|=|s(x)|$. Since
$$|s(x)+s(x)+{s(\mu)}|\ge |s(x)+s(x)| \ge |s(x)|$$ we have  $|s(x)+s(x)| = |s(x)|$, so $s(x)$ is a coset by Proposition \ref{veryeasy},
say $s(x)=g+U$. Then $$|U|=|s(x)|=|s(Vx)|=|g+U+g+U+{s(\mu)}|=|U+{s(\mu)}|,$$ so $U$ is ${s(\mu)}$-invariant.
\end{proof}

\begin{rk}
\label{remark}
The previous proposition implies that for every $x \in {S^G}$ there exists some $n_0 \in {\mathbb{N}}_0$ such that $|s(V^{n+1}x)|=|s(V^nx)|$
for all $n \ge n_0$ and $|s(V^{n-1}x)|<|s(V^nx)|$ for all $n\le n_0$. Further, $s(V^nx)$ is the coset of an  ${s(\mu)}$-invariant subgroup
$U$ iff $n \ge n_0$ (note that $U$ does not depend on $n$). Note that $s(V^nx)$ depends on $x$ only via $s(x)$ and on $\mu$ only via
its support ${s(\mu)}$.
\end{rk}

\begin{defn}
For a nonempty subset $B$ of $G$ the {\em uniform distribution} $u(B) \in {S^G}$ on $B$ is defined as
$$
u(B)_k=\left\{ \begin{array}{ll}
1/|B|, & k \in B,\\
0, & k \notin B.
\end{array}\right.
$$
\end{defn}

{The next theorem is our key result.}

\begin{thm}
\label{convergence}
For each $x \in {S^G}$ we have
$$
\lim_{n \to \infty} \big(V^nx-u(s(V^nx))\big)= 0.
$$
\end{thm}

\begin{proof}
For given $x \in {S^G}$ and $n_0$ as in Remark \ref{remark}, let $k=|s(V^{n_0}x)|$. The case $k=1$ is clear, so we assume that $k>1$.
Assume that the support of  $y \in {S^G}$ equals a coset of an ${s(\mu)}$-invariant subgroup $U$ of cardinality $k$ 
(as is the case for $V^nx$ whenever $n \ge n_0$) and let $v_1 \ge v_2 \ge {\cdots \ge \,} 
v_k$ be the numbers $y_j \in s(y)$ in decreasing order. 
{Then, the well known {\em rearrangement inequality} gives
$$
\sum\limits_{r=1}^k v_r v_{k+1-r}\leq \sum\limits_{i \in s(y)} y_i y_{\sigma(i)} \leq \sum\limits_{r=1}^k v_r^2,
$$ 
which holds for any permutation $\sigma$ of $s(y)$. Combined with (\ref{Gmkso}), this gives for every} $j \in s(Vy)$
\begin{align}
\label{permu}
(Vy)_j &=\sum_{l\in j-i-s(y)} \mu_l \sum_{i \in s(y)} y_i y_{j-l-i} \notag \\
&\ge \sum_{r=1}^k v_r v_{k+1-r} \notag \\ 
&{= v_k\sum_{r=1}^k v_r + \sum_{r=1}^k (v_{k+1-r} -v_k)v_r} \notag \\
&\ge v_k\sum_{r=1}^k v_r+(v_1-v_k)v_k \notag \\
&=v_k+(v_1-v_k)v_k. 
\end{align}
In particular, the smallest element of {$(Vy)$} 
is strictly larger than the smallest element
of {$y$} 
unless $y=u(s(y))$. Hence, by \eqref{permu},
\begin{equation}
\label{eq:22}
{\Big(0, \frac 1k\Big] \, \ni} \, \alpha=\min_{j \in s(Vy)}(Vy)_j \ge v_k+(v_1-v_k)v_k.
\end{equation}
{Note further that $1 = \sum_{j=1}^k v_j \le (k-1)v_1+v_k$, so that
\begin{align*}
(v_1-v_k)v_k &= \frac{(k-1)v_1v_k}{k-1} - \frac{(k-1)v_k^2}{k-1}\\
&\ge \frac{(1-v_k)v_k}{k-1} - \frac{(k-1)v_k^2}{k-1}\\
&= \frac{v_k}{k-1} - \frac{kv_k^2}{k-1} = \Big(\frac 1k -v_k\Big) \frac{kv_k}{k-1}.
\end{align*}
Putting this and (\ref{eq:22}) together gives}
$$
\frac 1k-\alpha \le \frac 1k - v_k-(v_1-v_k)v_k \le \Big(\frac 1k-v_k\Big)\Big(1-\frac{kv_k}{k-1}\Big)
$$
which immediately implies the statement in the theorem.
\end{proof}

\begin{rk}
If ${k:=}|s(x)|>m/2$, then {the only subgroup of $G$ with order at least $k$ is $G$ itself, hence} $V^n(x)$ converges to the center ${u(G)}$ of ${S^G}$ thus showing that $V$ is {\em almost regular}. Further, Theorem \ref{convergence} 
shows that $V$ is ergodic.
\end{rk}

%

The following theorem shows that the speed of convergence in the previous theorem is {double-exponential, i.e.\ rather} fast. {Let $\|{\cdot}\|$ denote} an arbitrary norm on {$\R^G$}.

\begin{thm} For each $x \in {S^G}$,
$$
\limsup_{n \to \infty}\frac 1n \log\log\|V^nx-u(s(V^nx))\|\le \log 2.
$$
\end{thm}

\begin{proof}
For given $x \in {S^G}$ and $n_0$ as in Remark \ref{remark}, let $k=|s(V^{n_0}x)|$. If the support of  $y \in S$ equals
a coset of an ${s(\mu)}$-invariant subgroup $U$ of cardinality $k$ (as is the case for $V^nx$ whenever $n \ge n_0$), then, denoting
$$
\varepsilon=\max_{{i} \in s(y)} {\Big\{ \Big|} y_{i}-\frac 1k {\Big|\Big\}},
$$ 
we have for $j \in s(Vy)$
$$
(Vy)_j-\frac 1k=\sum_{i \in s(y)}\sum_{l\in j-i-s(y)} \mu_l{\Big(}y_i-\frac 1k{\Big)}{\Big(}y_{j-l-i}-\frac 1k{\Big)},
$$
and therefore
$$
{\Big|}(Vy)_j-\frac 1k{\Big|} \le \varepsilon^2 k.
$$
Using Theorem \ref{convergence} the claim follows.
\end{proof}

We also have the following kind of converse to Theorem \ref{convergence}. 

\begin{thm}
\label{thm:periodic}
For any ${s(\mu)}$-invariant subgroup $U$ and any {$a \in s(\mu)$ and} $g \in G$ for which $n \mapsto 2^ng, { n \in \mathbb{N}}$ is periodic, {the sequence of uniform distributions $u(2^ng-a+U), n \in \mathbb{N},$} is a periodic orbit
of the map $V$.
\end{thm}

\begin{proof}
Let $U$ be an any ${s(\mu)}$-invariant subgroup and {$a \in s(\mu)$. Assume that the mapping $n \mapsto 2^n g$ is periodic. Then, for 
$x=u(g-a+U)$ we obtain from (\ref{eq:gUgU}) that
$$
V^n(x) = u (2^n g - a + U)
$$
which is periodic in $n$.}
\end{proof}


One can ask under which conditions on $G$ and ${s(\mu)}$ the QSO (\ref{Gmkso}) is regular.

\begin{pro}
\label{regular}
{A} $(G,\mu)-$QSO is regular iff for each ${s(\mu)}$-invariant subgroup $U$ and every $g \in G$ there exists some $n \in {\mathbb{N}}$ such that $2^ng \in U$.
In this case, all trajectories converge to the uniform distribution on some coset of some ${s(\mu)}$-invariant subgroup $U$.
\end{pro}

\begin{proof}
{
Theorem \ref{convergence} and Theorem \ref{thm:periodic} show that all periodic (possibly constant) trajectories are of the form $u(2^n g-a+U)$ for some $a \in s(\mu), g \in G$ and $U$ an $s(\mu)$-invariant subgroup of $G$, and that, conversely, all such trajectories are periodic. It remains to investigate which of these trajectories have a minimal period at least 2. 

Suppose we are given a periodic trajectory corresponding to $g, a, U$ as above.
If - on the one hand - all $2^n g, n \in \mathbb{N}$ eventually end up in $U$, then there is no strict periodicity, i.e.\ the trajectory becomes constant.

If - on the other hand - $g, a, U$ are such that $2^n g \notin U$, then 
$$
2^{n+1}g-a+U \neq 2^n g - a+U.
$$
Consequently, if $2^n g \notin U$ holds for all $n \in \mathbb{N}$, then  $V$ is not regular.
The last statement follows from Theorem \ref{convergence}.
}
\end{proof}

{Finally}, we state some more explicit conditions for regularity or non-regularity.

\begin{cor}
If $m{=|G|}=2^k$ for some $k \in  {\mathbb{N}_0}$, then {the} $(G,\mu)-$QSO is regular for any choice of $\mu$.
\end{cor}
\begin{proof}
In this case, the order of any $g \in G$ is of the form $2^{\nu}$ for some $\nu  {\le k}$, so the claim follows from Proposition \ref{regular}.
\end{proof}

\begin{cor}
If $m{=|G|}$ is not a power of 2, then there exists some $\mu \in {S^G}$ for which {the} $(G,\mu)-$QSO is not regular, e.g. $\mu \in {S^G}$ defined by $\mu_0=1$.
\end{cor}
\begin{proof}
Since $m$ is not a power of 2 there exists some prime number $p\neq 2$ and some $g \in G$ of order $p$ (this follows from the fact that
every finite Abelian group is the sum of cyclic groups whose orders are powers of primes). Then $2^ng$ is different from 0 for any $n \in {\mathbb{N}_0}$.
Therefore, $x \in {S^G}$ defined as $x_h=\delta_{gh}$ is the initial point of a trajectory of $V$ of minimal period $p$. In particular, $V$ is not
regular.
\end{proof}

\section*{ Acknowledgements}
The {second a}uthor (U.J.) thanks {the} IMU Berlin Einstein Foundation Program (EFP) {and} Berlin Mathematical School (BMS) for {a} scholarship and for support{ing} his visit to Technische Universi{t\"at} (TU) Berlin and TU Berlin for kind hospitality.\\

\end{document}